\documentclass[12pt]{article}
\usepackage{amsmath,amssymb,amsthm, amsfonts}
\usepackage{hyperref}
\usepackage{graphicx}
\usepackage{array, tabulary}
\usepackage{url}
\usepackage[mathlines]{lineno}
\usepackage{dsfont} 
\usepackage{color}
\usepackage{subcaption}
\usepackage{enumitem}
\definecolor{red}{rgb}{1,0,0}

\definecolor{blue}{rgb}{0,0,1}

\definecolor{green}{rgb}{0,.6,0}

\usepackage{float}

\usepackage{tikz}

\newtheorem{thm}{Theorem}[section]
\newtheorem{cor}[thm]{Corollary}

\newtheorem{prop}[thm]{Proposition}
\newtheorem{conj}[thm]{Conjecture}

\theoremstyle{definition}

\theoremstyle{definition}
\newtheorem{defn}[thm]{Definition}

\theoremstyle{definition}

\newcommand{\dist}{\operatorname{dist}}

\newcommand{\bit}{\begin{itemize}}
\newcommand{\eit}{\end{itemize}}
\newcommand{\ben}{\begin{enumerate}}
\newcommand{\een}{\end{enumerate}}
\newcommand{\beq}{\begin{equation}}
\newcommand{\eeq}{\end{equation}}
\newcommand{\bea}{\begin{eqnarray}} 
\newcommand{\eea}{\end{eqnarray}}
\newcommand{\bpf}{\begin{proof}}
\newcommand{\epf}{\end{proof}\ms}
\newcommand{\bmt}{\begin{bmatrix}}
\newcommand{\emt}{\end{bmatrix}}
\newcommand{\ms}{\medskip}

\newcommand{\noi}{\noindent}

\newcommand{\beqs}{\begin{equation*}} 
\newcommand{\eeqs}{\end{equation*}}
\newcommand{\beas}{\begin{eqnarray*}}
\newcommand{\eeas}{\end{eqnarray*}}

\newcommand{\dmg}{\operatorname{dmg}}

\title{The multi-robber damage number of a graph}
\author{Joshua Carlson
\thanks{Department of Mathematics and Computer Science, Drake University, Des Moines, IA, USA (joshua.carlson@drake.edu)}
\and Meghan Halloran
\thanks{Department of Mathematics and Statistics, Williams College, Williamstown, MA, USA (meghanhalloran7@gmail.com)}
\and Carolyn Reinhart
\thanks{Department of Mathematics and Statistics, Swarthmore College, Swarthmore, PA, USA (creinha1@swarthmore.edu)}}

\date{\today}

\begin{document}

\maketitle

\begin{abstract}
In many variants of the game of Cops and Robbers on graphs, multiple cops play against a single robber. In 2019, Cox and Sanaei introduced a variant of the game that gives the robber a more active role than simply evading the cop. In their version, the robber tries to damage as many vertices as possible and the cop attempts to minimize this damage. While the damage variant was originally studied with one cop and one robber, it was later extended to play with multiple cops by Carlson et.~al in 2021. We take a different approach by studying the damage variant with multiple robbers against one cop. Specifically, we introduce the $s$-robber damage number of a graph and obtain a variety of bounds on this parameter. Applying these bounds, we determine the $s$-robber damage number for a variety of graph families and characterize graphs with extreme $2$-robber damage number.
\end{abstract}

\noi {\bf Keywords} Cops and Robbers, Damage number

\noi{\bf AMS subject classification} 05C57, 05C15, 05C50

\section{Introduction}

Cops and Robbers is a perfect information pursuit-evasion game played on simple graphs that was introduced in \cite{NW83, Q78}. Originally, the game was played with two players (cop and robber) that move from vertex to vertex by traversing the edges of the graph. The game is initialized in round $0$ when (starting with the cop) both players choose an initial vertex to occupy. Then, each subsequent round consists of a turn for the cop followed by a turn for the robber where each player has the opportunity to (but is not required to) move to a neighboring vertex on their turn. Of course, if the cop ever occupies the same vertex as the robber, the robber is said to be \emph{captured} and the game ends in victory for the cop. Alternatively, if the robber has a strategy to avoid capture forever, the robber wins the game.

In \cite{AF84}, the authors consider a version of the game with more players. Specifically, a team of $k$ cops plays against a single robber. In this version, each round consists of a turn for the team of cops followed by a turn for the robber where on the cops turn, each cop has the opportunity to move. As in the original game, in round $0$, each cop chooses their initial position before the robbers' position is initialized. This multi-cop version of the game leads to the main parameter of interest in the study of cops and robbers. The \emph{cop number} of a graph $G$, denoted $c(G)$, is the smallest number of cops required for the cop team to guarantee capture of the robber on $G$.

There are many variations of cops and robbers that have been studied in which it is interesting to consider multiple players on the cop team (see \cite{AF84, BMPP16, BPPR17, FHMP16}). Other variants slightly alter the objectives of the players. One such version, introduced in \cite{CS19}, states that if a vertex $v$ is occupied by the robber at the end of a given round and the robber is not caught in the following round, then $v$ becomes \emph{damaged}. In this version of the game, rather than trying to capture the robber, the cop is trying to minimize the number of damaged vertices. Additionally, the robber plays optimally by damaging as many vertices as possible.

The damage variation of cops and robbers leads to another parameter of interest. The \emph{damage number} of a graph $G$, denoted $\dmg(G)$, is the minimum number of vertices damaged over all games played on $G$ where the robber plays optimally. Although the damage variant was introduced with a singe cop and robber, in \cite{CEGPRS21}, the authors extended the idea of damage to games played with $k$ cops against one robber. Specifically, they introduce the \emph{$k$-damage number} of a graph $G$, denoted $\dmg_k(G)$, which is defined analogously to $\dmg(G)$. 

Note that when the goal of the cops is simply to capture the robber, there is no reason to add players to the robber team because a strategy of the cop team to capture one robber is sufficient for repeatedly capturing additional robbers. However, in the damage variant, it the robber who is the more active player since their goal is to damage as many vertices as possible. This creates a somewhat rare situation where it becomes interesting to play with multiple robbers and one cop. We now generalize the damage number in a new way with the following definition.

\begin{defn}
Suppose $G$ is a simple graph. The \emph{$s$-robber damage number} of $G$, denoted $\dmg(G;s)$, is the minimum number of vertices damaged in $G$ over all games played on $G$ where $s$ robbers play optimally against one cop. Note that optimal play for the robbers is still to damage as many vertices as possible.
\end{defn}

The $s$-robber damage number is the main focus of this paper. All graphs we consider are finite, undirected, and simple. We adhere to most of the graph theoretic and Cops and Robbers notation found in \cite{Diestel} and \cite{CRbook} respectively. In Section \ref{sec:generalBounds}, we establish some general bounds on $\dmg(G;s)$ in terms of the number of vertices and the number of robbers. We focus on $\dmg(G;2)$ in Section \ref{subsec:2generalBounds}, providing an upper for graphs with maximum degree at least three. Then, in Section \ref{sec:srobberFamilies}, we determine $\dmg(G;s)$ for various graph families, including paths, cycles, and stars. Finally, in Section \ref{sec:extreme2robber}, we characterize the graphs with extreme values of $\dmg(G;2)$. Interestingly, we show that threshold graphs are exactly the graphs with $\dmg(G;2)=1$.

\section{General results on the $s$-robber damage number}\label{sec:generalBounds}

We begin by establishing bounds on the $s$-robber damage number. Throughout this section, we find upper bounds by describing a cop strategy which limits damage to some number of vertices and we find lower bounds by describing a robber strategy for which some number of vertices are always damaged. First, we find a general lower bound for all graphs on $n$ vertices.

\begin{prop}\label{prop:damageAtLeastSMinus1}
Suppose $G$ is a graph on $n$ vertices. If $s\leq n-1$, then $\dmg(G; s) \geq s-1$ and if $s\geq n$, then $\dmg(G; s) \geq n-2$.
\end{prop}
\begin{proof}
Let the cop start on any vertex $v$. If $s\leq n-1$, place all of the robbers on separate vertices in $V(G) \setminus \{v\}$. The cop can only capture at most 1 robber in the first round, therefore at least $s-1$ vertices will be damaged. If $s\geq n$, then place at least one robber on each vertex of $V(G) \setminus \{v\}$. In the first round, if the cop moves to capture a robber, they can prevent damage to at most one vertex in $V(G) \setminus \{v\}$. The only other vertex which will not be damaged in the first round is $v$. Therefore, at least $n-2$ vertices will be damaged.
\end{proof}

We now provide a lower bound for all graphs on $n\geq 2$ vertices with at least one edge. Note that we later compute the $s$-robber damage number of the empty graph in Proposition \ref{prop:Empty}.

\begin{prop}\label{prop:damageAtMostNMinus2}
Suppose $G$ is a graph on $n \geq 2$ vertices with at least 1 edge. Then $\dmg(G; s) \leq n-2$ for each $s \geq 1$.
\end{prop}
\begin{proof}
Consider a cop strategy where the cop starts on a vertex $v$ with positive degree and toggles between $v$ and one of its neighbors $u$. If the robber moves to $u$ or $v$, the cop either captures the robber immediately or moves to capture the robber in the following round. Since the cop can prevent at least two vertices from being damaged, $\dmg(G; s) \leq n-2$. 
\end{proof}

The combination of Propositions \ref{prop:damageAtLeastSMinus1} and \ref{prop:damageAtMostNMinus2} yields an immediate corollary in the case where the number of robbers is at least the number of vertices.

\begin{cor}
Suppose $G$ is a graph on $n \geq 2$ vertices with at least 1 edge. If $s\geq n$, then $\dmg(G; s) = n-2$. 
\end{cor}

Since we are considering graphs which are not necessarily connected, it is useful to compute the $s$-robber damage number of the disjoint union of graphs. In the case of a graph with two disjoint components, we can compute the $s$-robber damage number as follows.
\begin{prop}
For $s \geq 1$ and graphs $G$ and $H$, let $\ell = \max\{\dmg(G;s-1) + |H|, \dmg(G;s)\}$ and $r = \max\{\dmg(H;s-1) + |G|, \dmg(H;s)\}$. Then, $\dmg(G \cup H; s) = \min \{ \ell, r\}$ .
\end{prop}
\begin{proof}
Suppose the cop starts on $G$. If $\dmg(G; s) > \dmg(G;s-1) + |H|$, then the robbers' strategy will be to all start on $G$ and damage $\dmg(G; s)$ vertices. Otherwise, at least one robber should start on $H$. However, since the cop is not on $H$, one robber in $H$ is enough to damage all $|H|$ vertices. So the remaining $s-1$ robbers should choose to start on $G$ and $\dmg(G;s-1) + |H|$ will be damaged. Therefore, if the cop starts on $G$, $\ell$ vertices are damaged. Similarly, if the cop starts on $H$, $r$ vertices are damaged. Since the cop is playing optimally, the cop will start on whichever graph will yield the least damage. Therefore, $\dmg(G \cup H; s) = \min \{\ell,r\}$.
\end{proof}

Finally, we consider graphs containing cut vertices and determine upper and lower bounds in terms of $s$ and the number of connected components which result from removing a cut vertex.

\begin{prop}
For a graph $G$, if there exists a vertex $v\in V(G)$ such that $G-v$ has $k \geq 1$ non-trivial connected components, then $\dmg(G,s)\geq \min(2k-2,2s-2)$ for all $s\geq 1$.
\end{prop}
\begin{proof}
Let $v \in V(G)$ such that $G-v$ has $k$ non-trivial components. Label the components $C_1,\dots, C_k$. Observe that for vertices $v_i$ and $v_j$ which are in different non-trivial components, $\dist(v_i,v_j)\geq 2$. If $s\geq k$, at least one robber can start in each of the $k$ non-trivial components. If the cop captures a robber in $C_i$ on round 1, it will be at least round 3 before a robber in $C_j$ for $i\not=j$ is captured. Since component $C_j$ is non-trivial, the robber(s) in this component can damage vertices on both rounds 1 and 2. So two or more vertices are damaged in every component except for the component in which the cop captured a robber in round 1. Thus, $\dmg(G;s)\geq 2k-2$. If $s<k$, then each robber starts on a different connected component, say $C_1,\dots, C_s$. Using the same strategy as in the previous case, all the robbers except for the one captured first can damage at least two vertices. Thus, $\dmg(G,s)\geq 2s-2$.
\end{proof}

\begin{prop} \label{damage at most n-d}
If there exists a vertex $v \in V(G)$ such that $G-v$ has $k\geq 1$ connected components, then $\dmg(G; s) \leq \min(n-k+s-2, n-2)$ for all $s\geq 1$. 
\end{prop}
\begin{proof}
Let $v \in V(G)$ such that $G-v$ has $k$ components. First, assume $s\leq k$ and label $s$ of the components $C_1,\dots,C_s$ and the rest of the components (excluding $v$), $C$. Note that $|C| \geq k-s$. Suppose the cop starts on $v$ and suppose one robber starts on each of the components $C_1,\dots,C_s$. Choose a neighbor of $v \in C_1$ and call this vertex $w$. Let the cop protect the edge $vw$ by moving between $v$ and $w$. This implies that the cop can protect all of the vertices in $C$ in addition to $v$ and $w$. Therefore, the cop can protect at least $k-s+2$ vertices, so $\dmg(G; 2) \leq n-k+s-2$. If $s > k$, then $\dmg(G;s) \leq n-2$ by Proposition \ref{prop:damageAtMostNMinus2}.
\end{proof}

\subsection{A bounds on the $2$-robber damage number}\label{subsec:2generalBounds}

We now turn our focus to the case where $s=2$. In the next result, we consider graphs which contain a vertex of degree at least three and show that in this case, the bound from Proposition \ref{prop:damageAtMostNMinus2} can be improved from $n-2$ to $n-3$.

\begin{prop} \label{prop:maxDegreeThree}
For a graph $G$ on $n$ vertices, if $\Delta(G)\geq 3$, then $\dmg(G; 2) \leq n-3$.
\end{prop}
\begin{proof}
Consider a graph $G$ with $\Delta(G)\geq 3$ and let $v$ be a vertex with at least 3 neighbors $x, y, z \in V(G)$. Let the cop's strategy be to start on $v$ and try to protect $x, y, z$. This implies that the robbers can move freely on the other vertices, but the cop only reacts when one or both robbers move to $x, y, z$ or $v$. Therefore, we only need to consider the subgraph induced by these 4 vertices, which we call $N$. 


Let the robbers be $R_1$ and $R_2$, and first suppose at most one robber ever moves to a vertex in $N$. If a robber moves to $N$, the cop can clearly capture them, so no vertices in $N$ are damaged. Next, suppose both robbers move to $N$ at some point during the game. If the robbers move to $N$ in non-consecutive rounds, it is clear that the cop can capture the first robber and then return to $v$. When the second robber moves to $N$ the cop can capture them too, thus protecting all $4$ vertices in $N$. Suppose the robbers show up in consecutive rounds. Without loss of generality, let $R_1$ move to $x$. In the next round, the cop will move from $v$ to $x$ to capture $R_1$ and $R_2$ will move to a vertex in $N$. If $R_2$ moved to $v$, then the cop can move back to $v$ and capture in the next round, so no vertices of $N$ are damaged. Otherwise, $R_2$ moved to $y$ or $z$, without loss of generality, say $y$. After capturing $R_1$, the cop will move back to $v$, protecting $x, z$ and $v$ and $R_2$ will damage $y$. No matter where $R_2$ moves next, the cop can still protect $x, z$ and $v$ from becoming damaged. 

Finally, suppose both robbers move to $N$ in the same round. In this case, the cop's strategy depends on the edges between $x, y,$ and $z$. First, suppose there are no edges between $x, y,\text{ or } z$. The cop can follow a similar strategy to the previous one. Without loss of generality, let $R_1$ move to $x$ and let $R_2$ move to $y$. The cop will move to $x$ in the next round to capture $R_1$ and $R_2$ will damage $y$. Next, $R_2$ can either move to $v$ or leave $N$ and the cop will return to $v$. From here it is clear that $R_2$ will not damage another vertex in the next round and if $R_2$ ever re-enters $N$ it is clear that the cop can capture them. Therefore the cop has prevented $v, x,$ and $z$ from being damaged. 

Next, suppose there exists one edge within ${x, y, z}$ and without loss of generality we'll assume the edge is between $x$ and $y$. If $R_1$ and $R_2$ move to $x$ and $y$, then the cop will move to $x$ to capture $R_1$. At this point, $R_2$ has damaged $y$ and can either move to $x$, $v$ (in either case, the cop can capture), or leave $N$. So it is clear that the cop can prevent $v, x,$ and $z$ from being damaged. If one robber moves to a vertex on the edge $xy$ and one robber moves to $z$, the cop will have a different strategy. Suppose $R_1$ moves to $z$ and $R_2$ moves to $y$. The cop will move to $y$, capturing $R_2$, and $R_1$ will damage $z$. From here, the cop can return to $v$ and protect $v, x$ and $y$ the rest of the game. 

Now, suppose there exists two edges within $x, y, z$. Without loss of generality, we'll let the edges be $xz$ and $yz$. First, suppose one robber moves to $z$ and the other moves to $x$ or $y$. We'll let $R_1$ move to $z$ and $R_2$ move to $x$. The cop can move to $z$ to capture $R_1$ and $R_2$ will damage $x$. From here, the cop can protect the vertices neighboring $x$ within $N$. This implies that $R_1$ cannot damage anymore vertices within $N$. Next, suppose neither robber moves to $z$ at first. We'll let $R_1$ move to $x$ and $R_2$ move to $y$. The cop will move to $x$ to capture $R_1$ and $R_2$ will damage $y$. From here, the cop will be able to protect the neighbors of $y$ within $N$ ($z$ and $v$), therefore preventing $R_2$ from damaging anymore vertices within $N$. 

Finally, suppose there exists an edge between each pair of neighbors of $v$ in $N$. This implies that $N$ is $K_4$, so the cop can capture one robber each round, and only one vertex will be damaged within $N$. We have shown that for all cases, the cop can prevent at least 3 vertices from being damaged, therefore $\dmg(G; 2) \leq n-3$. 
\end{proof}

Next, it is natural to ask whether Proposition \ref{prop:maxDegreeThree} can be generalized for all $s$ and $n \geq 1$. The most obvious generalization would be: if $\Delta(G) \geq s+1$, is $\dmg(G; s) \leq n-s-1$? We can use Proposition \ref{prop:damageAtLeastSMinus1} to answer this question negatively in the following way. Note that if $n < 2s$, then $n-s-1 < s-1$. Thus, by Proposition \ref{prop:damageAtLeastSMinus1}, $\dmg(G; s) \geq s-1 > n-s-1$. Therefore, it is possible to have a graph on $n > 2s$ vertices with $\Delta(G) \geq s+1$ such that $\dmg(G; s) > n-s-1$. An example of this is illustrated in Figure \ref{fig:wheelOn5Vertices}.

\begin{figure}[h] \begin{center}
\scalebox{.8}{\includegraphics{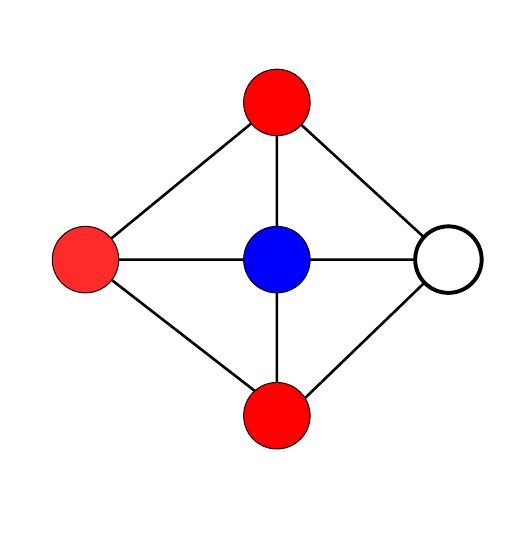}}\\
\caption{The wheel on 4 vertices has $\dmg(W_4; s) > n-s-1$ for $s \in \{3, 4\}$. An initial placement with 1 cop (in blue) and 3 robbers (in red) is shown above.}\label{fig:wheelOn5Vertices}
\end{center}
\end{figure}

We now consider another possible generalization. The following conjecture maintains the upper bound of $n-3$, but generalizes the condition on the maximum degree that is required. 

\begin{conj}\label{conj:maxdeg}
In a graph $G$, if  $\Delta(G)\geq\binom{s}{2}+2$, then $\dmg(G; s) \leq n-3$ for all $s \geq 2$.
\end{conj}

\section{The $s$-robber damage number of graph families}\label{sec:srobberFamilies}


In this section, we determine the $s$-robber damage number for certain graph families. We begin by considering the empty graph $\overline{K_n}$ and the complete graph $K_n$ on $n$ vertices.

\begin{prop}\label{prop:Empty}
For $n\geq 1$, $\dmg (\overline{K_n}; s) = \min\{s, n-1\}$ for all $s\geq 1$. 
\end{prop}
\begin{proof}
Let $1 \leq s \leq n-1$ and suppose the cop starts on vertex $v \in V(G)$. The robbers can each start on distinct vertices in $V(G) \setminus \{v\}$ and the cop can only protect $v$. Thus, $s$ vertices are damaged. If $s > n-1$, let the $s$ robbers start on the $n-1$ vertices not occupied by the cop. Therefore, $n-1$ vertices are damaged.
\end{proof}

\begin{prop}
For $n \geq 4$, $\dmg(K_n; s) = \min\{\frac{s(s-1)}{2}, n-2\}$ for all $s\geq 1$.
\end{prop}
\begin{proof}
First, note that by Proposition \ref{prop:damageAtMostNMinus2}, $\dmg(K_n; s) \leq n-2$. Next, we assume $\frac{s(s-1)}{2}\leq n-2$ and show that there exists a cop strategy such that $\dmg(K_n; s) \leq \min\{\frac{s(s-1)}{2}\}$. Since every vertex in $K_n$ is a dominating vertex, the cop can capture a new robber each round until all of the robbers have been caught. Since $\binom{s}{2} \leq n-2$, in the first round, $s-1$ vertices will be damaged and as the cop continues to capture robbers, $s-2, s-3, ...$ vertices will be damaged each round. Therefore, if there are enough vertices in the graph, the robbers can damage at most $(s-1) + (s-2) + ... = {s \choose 2} = \frac{s(s-1)}{2}$ vertices. Thus, the cop should use this strategy when $\frac{s(s-1)}{2} \leq n-2$ and use the strategy from Proposition \ref{prop:damageAtMostNMinus2} otherwise. This implies that $\dmg(K_n; s) \leq \min\{\frac{s(s-1)}{2}, n-2\}$.

Next, we will give a strategy for the robbers such that no matter what the cop does, the robbers can damage at least $\min\{\frac{s(s-1)}{2}, n-2\}$ vertices. Let the robbers start on as many vertices as possible, but not the vertex that the cop starts on. If ${s \choose 2} \leq n-2$, all of the robbers can start on distinct vertices and it is clear that the cop can only capture one robber in the first round. This implies that after the first round, $s-1$ vertices are damaged and $s-1$ robbers remain uncaught. Suppose the robbers try to damage as many vertices as possible by moving to different undamaged vertices each round. Thus, the robbers can damage $(s-1) + (s-2) +... = \frac{s(s-1)}{2}$ vertices, no matter what the cop does. 

Now, suppose ${s \choose 2} > n-2$. This implies that at some point in the game, the number of undamaged vertices, $k$, is less than the number of remaining robbers. Assuming the cop has been playing optimally up to this point, the cop will be occupying one of these undamaged vertices. 
Therefore, by moving to the undamaged vertices, the robbers can damage at least $k-2$ vertices in the next round. This leaves 2 vertices undamaged, which implies that the robbers can damage at least $n-2$ vertices. Therefore, we have established that $\dmg(K_n; s) = \min \{\frac{s(s-1)}{2}, n-2\}$.
\end{proof}


We next consider the path graph on $n$ vertices, $P_n$ and show that for any number of robbers $s$, the $s$-robber damage number is $n-2$.

\begin{thm}\label{thm:path}
For $n, s \geq 2$, $\dmg(P_n; s) = n-2$.
\end{thm}
\begin{proof}
By Proposition \ref{prop:damageAtMostNMinus2}, we have that $\dmg(P_n; s) \leq n-2$. To show $\dmg(P_n; s) \geq n-2$, we argue that for any cop strategy, the robbers are able to damage $n-2$ vertices.

For $s> 2$, the robbers can form two non-empty groups such that every robber in each group acts as a single robber. Thus, it is sufficient to prove the result for $s=2$. Let the two robbers be called $R_1$ and $R_2$.

If $n=2$, it is clear that the cop can protect the two vertices and therefore the robbers are not able to damage any vertices. So, $n-2 = 2-2 = 0$ vertices can be damaged. Next, let $n > 2$. If the cop starts on a leaf, the robbers can start on the vertex which is distance two away from this leaf. On each round, the robbers can move towards the other end of the path and will not be captured until they reach the end. Therefore, the robbers can damage $n-2$ vertices.

Now, suppose the cop starts on a neighbor of a leaf. If $n=3$, the only neighbor of a leaf is the middle vertex and a robber can start on each leaf. Since the cop can only capture one of the robbers in the first round, it is clear that at least one vertex will be damaged and $n-2 = 3-2 =1$. If $n > 3$, place $R_1$ on the leaf neighboring the cop and place $R_2$ on the vertex of distance two from the cop. If the cop passes during the first round, $R_1$ will damage the leaf and $R_2$ can move to the other end of the path, damaging $n-3$ vertices. Therefore, $n-3+1 = n-2$ vertices are damaged. If the cop captures $R_1$ in the first round, then $R_2$ can move towards the cop in the first round and then move back towards the other end of the path, damaging $n-2$ vertices. If the cop moves towards $R_2$ in the first round, $R_2$ will move to the other end of the path, damaging $n-3$ vertices on the way. Since $R_1$ will at least damage one vertex (the leaf), at least $n-3+1 = n-2$ vertices are damaged.

Finally, suppose the cop starts on a vertex which is distance at least two from both leaves. It is clear in this case that $n\geq 5$. Consider the cop's initial vertex and the two vertices to its left and right. We label these vertices $v_1,...,v_5$, left to right, so the cop starts on $v_3$. Let $R_1$ start on $v_1$ and $R_2$ start on $v_5$. Let $x$ and $y$ be the number of vertices in $P_n$ to the left of $v_1$ and to the right of $v_5$, respectively. Without loss of generality, suppose $x \leq y$ (note that $x$ or $y$ could be zero).

If the cop moves to $v_2$ in the first round, then the robbers will both move to the left as well and $R_2$ will damage $v_4$. Similarly, if the cop moves to $v_4$ in the first round, then the robbers will both move to the right as well and $R_1$ will damage $v_2$. After this happens, $R_1$ can move left during every turn and $R_2$ can move right during every turn (until they reach a leaf), damaging each vertex on their path. It is clear that $v_3$ and the vertex the cop moves to in the first round are the only undamaged vertices. Therefore, $n-2$ vertices will be damaged.

If the cop doesn't move first, then the robbers must move first (otherwise, if neither player moves, only two vertices are damaged). It is obvious that $R_1$ can damage $x+1$ vertices without being caught. As $R_1$ is damaging those vertices, $R_2$ can stay exactly two vertices to the right of the cop, whenever possible.

If $R_2$ is ever captured, this strategy ensures capture will occur on the right leaf. Capturing $R_2$ on that vertex will take the cop at least $2+y$ rounds. In order to prevent damage to all of the vertices, the cop must then move back to $v_3$. Note that the cop requires at least $2(2+y) = 4 + 2y$ rounds to capture $R_2$ and return to $v_3$. However, in at most $2x+1$ rounds, $R_1$ can move left, damaging the left side of the path, and then return to $v_2$. Since $x \leq y$, it's clear that $2x + 1 < 2y + 4$, which means $R_1$ can damage $v_2$. Overall, $R_1$ can damage at least $x+2$ vertices and $R_2$ can damage $y+1$ vertices and therefore, at least $n-2$ vertices will be damaged.  

Otherwise, assume that $R_2$ is not captured. If the cop ever moves to the left of $v_3$ towards $R_1$, then $R_2$ can damage $v_4$, $v_5$ and the $y$ vertices to the right $v_5$ without being caught. It is clear that $v_2$ and $v_3$ are the only undamaged vertices, so $n-2$ vertices can be damaged. Next, suppose the cop never moves to the left of $v_3$. If the cop is to the right of $v_3$ when $R_1$ returns to $v_1$, it's clear that $R_1$ can damage $v_2$. At this point, $R_2$ can damage any remaining vertices on the right side of the path, so $x+2+y+1=n-2$ vertices can be damaged. If the cop is on $v_3$ when $R_1$ returns to $v_1$, $R_2$ is on $v_5$. If the cop moves to either $v_2$ or $v_4$, then the robbers can act as if the cop did this in round one, and damage $n-2$ vertices as in that case. If the cop passes, $R_1$ can move to $v_2$ and $R_2$ can stay on $v_5$. If the cop doesn't capture $R_1$, then $v_2$ will be damaged and $R_2$ can damage $v_5$ and $y$ more vertices without being caught, so $n-2$ vertices are damaged. On the other hand, if the cop moves to $v_2$ to capture $R_1$, then $R_2$ can move to $v_4$ and then move back down the right end of the path without getting caught. Therefore $n-2$ vertices are damaged.

We have shown that at least $n-2$ vertices are damaged regardless of what strategy the cop uses, so $\dmg(P_n; s) = n-2$. 
\end{proof}

Next, we show that $n-2$ is also the $s$-robber damage number for the cycle graph $C_n$ on $n$ vertices, employing a similar technique to Theorem \ref{thm:path}.

\begin{thm}\label{thm:cycle}
For $n \geq 3$ and $s \geq 2, \dmg(C_n; s) = n-2$.
\end{thm}
\begin{proof}
By Proposition \ref{prop:damageAtMostNMinus2}, we have that $\dmg(C_n; s) \leq n-2$. To show $\dmg(C_n; s) \geq n-2$, we argue that for any cop strategy, the robbers are able to damage $n-2$ vertices.

As in the proof of Theorem \ref{thm:path}, for $s> 2$, the robbers can form two non-empty groups such that every robber in each group acts as a single robber. Thus, it sufficient to prove the result for $s=2$. Let the two robbers be called $R_1$ and $R_2$.

If $n=3$, the robbers can start on the two vertices that the cop does not start on. In the first round, the cop can only capture one robber therefore one vertex will be damaged. Thus, damage is at least one. If $n = 4$, let $R_1$ start next to the cop and let $R_2$ start on the vertex of distance two from the cop. In the first round, the cop will capture $R_1$. Then $R_2$ can move to its neighbor that will be a distance of two away from the cop. This implies that $R_2$ can damage its starting vertex and a second vertex. Thus, at least two vertices are damaged.

If $n\geq 5$, suppose the cop starts on an arbitrary vertex $v_3$ and label the four closest vertices to $v_3$ as $v_1, v_2, v_4, v_5$, clockwise. Let the robbers, $R_1$ and $R_2$, start on vertices $v_1$ and $v_5$, respectively. Suppose there are $z=n-5$ vertices left unlabeled (note it is possible that  $z=0$). Split up the $z$ vertices into two sets, $X$ and $Y$, as follows. Let $X$ be the set of $\lceil \frac{n-5}{2} \rceil$ vertices, starting from the unlabeled neighbor of $v_1$ and moving counterclockwise. Similarly, let $Y$ be the set of $\lceil \frac{n-5}{2} \rceil$ vertices, starting from the unlabeled neighbor of $v_5$ and moving clockwise. Note that if $n$ is even, $X$ and $Y$ will both contain the vertex which is farthest away from $v_3$. 

Suppose the cop moves to $v_2$ in the first round. Then, $R_1$ will move in the same direction away from the cop and $R_2$ will move to $v_4$. At this point, $R_1$ and $R_2$ are guaranteed to damage $n-2$ vertices. This is because no matter what the cop does, $R_1$ and $R_2$ can move towards each other (and away from the cop), and damage the $z$ additional vertices without being caught. This implies that $z$ vertices plus $v_1, v_4,\text{ and } v_5$ are damaged, so $n-5 + 3 = n-2$ vertices are damaged. If the cop moves to $v_4$ in the first round, then the robbers can simply follow the same strategy with their roles reversed. 

If the cop passes on the first round, we can use a technique similar to the one in the proof of Theorem \ref{thm:path}. Let $R_1$ move counterclockwise, damaging the vertices in $X$, while $R_2$ stays a distance of two away from the cop. Using this strategy, it is clear that $R_2$ will not be captured. If the cop ever moves from $v_3$ to $v_2$, then we know that $R_2$ can damage $v_4$. Afterward, $R_2$ can move clockwise until the robbers have together damaged all remaining vertices. In this case, the robbers damage at least $z+3=n-2$ vertices.

If the cop never moves from $v_3$ to $v_2$, then the cop could only move to a vertex in $X$ by moving clockwise through $Y$. During this process, $R_2$ will stay a distance of two away from the cop and damage all of the vertices in $Y$, as well as $v_5$. It will take at least $\lceil \frac{n-5}{2} \rceil + 2$ rounds for the cop to enter $X$. However, $R_1$ can damage $v_1$ and all of the vertices in $X$ in $\lceil \frac{n-5}{2} \rceil + 1$ rounds. Then, $R_1$ can move clockwise back to $v_2$ without being captured, since the cop will always be at least distance two away. Thus, $n-2$ vertices are damaged.

If the cop never enters $X$, the cop will only ever move between the vertices in $Y \cup \{v_3, v_4, v_5\}$. This means that $R_1$ can damage $v_1$, $v_2$, and the vertices in $X$, since the cop will never enter these vertices. Meanwhile, $R_2$ can start moving clockwise on every turn while remaining at least distance two from the cop at all times. Using this strategy, $R_2$ can damage $v_5$ and the vertices in $Y$. Therefore, $n-2$ vertices are damaged.

We have shown that the robbers can damage at least $n-2$ vertices no matter what strategy the cop uses, so $\dmg(C_n; s) = n-2$. 
\end{proof}

Finally, we show that a similar technique to Theorem \ref{thm:path} can be used to compute the $s$-robber damage number of a spider graph.

\begin{thm}\label{thm:star}
Suppose $G$ is a spider graph with $\ell \geq 3$ legs of lengths $k_1\geq k_2\geq \dots\geq k_{\ell}$. If $2 \leq s\leq \ell$, $\displaystyle \dmg(G; s) =\left(\sum_{i=1}^s k_i\right) -1$ and if $s > \ell$, $\dmg(G; s) =n-2$ . 
\end{thm}
\begin{proof}
Let the vertex in the center of the spider be $c$. If $s > \ell$, the fact that $\dmg(G;s) \leq n - 2$ follows from Proposition \ref{prop:damageAtMostNMinus2}. If $2 \leq s\leq \ell$, suppose the cop starts on $c$ and remains there unless a robber moves to a neighbor of $c$. In this case, the cop's strategy will be to capture the robber and return back to $c$. This implies that if the robbers start on the $s$ longest legs, the cop can protect all of the other legs, as well as one vertex in a leg that contains a robber. Therefore, the cop can protect $n - \left(\sum_{i=1}^s k_i\right) + 1$ vertices and $\dmg(G; s) \leq \left(\sum_{i=1}^s k_i\right) -1$. 

If $s >l$, the robbers can behave as $\ell$ robbers which implies $\dmg(G; s)\geq \dmg(G; \ell)$. Since $(\sum_{i=1}^{\ell} k_i) -1 = n-2$, it is sufficient to assume $2 \leq s\leq \ell$ and provide a strategy for the robbers such that they can always damage at least $\left(\sum_{i=1}^s k_i\right) -1$ vertices for every cop strategy. We first consider the case where $k_i\geq 2$ for all $1\leq i\leq s$. Let $v_i$ be the vertex adjacent to $c$ in the leg of length $k_i$ for $1\leq i\leq \ell$, and let $u_i$ be the vertex adjacent to $v_i$ which is not $c$ for $1\leq i\leq s$. Call the $s$ robbers $R_1,R_2,\dots, R_s$.

Suppose the cop starts on $c$ and let $R_i$ place on $u_i$ for each $1\leq i\leq s$. If the cop moves in round one to $v_j$ for some $s+1\leq j\leq \ell$, each robber $R_i$ can move to $v_i$ and damage it. Then, regardless of what the cop does next, $R_i$ can move to the leaf in their leg without being captured. Thus, damage is at least $\left(\sum_{i=1}^s k_i\right)$. If the cop moves in round one to $v_j$ for some $1\leq j\leq s$, then $R_j$ will move towards the leaf in their leg and all the other robbers $R_i$ can move to $v_i$. On each subsequent round, regardless of what the cop does, each robber can move towards the leaf in their leg without being captured. Thus, at least $\left(\sum_{i=1}^s k_i\right)-1$ vertices are damaged.

If the cop passes during round 1, let $R_s$ move towards the leaf in its leg. While the cop remains on $c$, the other robbers should not move. If the cop ever moves from $c$ to $v_j$ for some $1\leq j\leq \ell$, all robbers $R_i$ for $i\not=s,j$ should move to $v_i$. In every round after this, each $R_i$ should move towards the leaf in their leg, damaging $k_i$ vertices. If $s\leq j\leq \ell$, then the robbers $R_1,\dots, R_{s-1}$ damage $\sum_{i=1}^{s-1} k_i$ vertices and $R_s$ damages $k_s-1$ vertices, so at least $\left(\sum_{i=1}^s k_i\right)-1$ vertices are damaged. 


If $1\leq j\leq s-1$, then $R_j$ should maintain a distance of two from the cop as long as they share a leg, or until $R_j$ is forced to the leaf of their leg and captured. If $R_j$ is captured, the cop will take at least $2k_j+1$ rounds to capture $R_j$ and return to the center (since the cop passed in the first round). However, $R_s$ can move to the end of their leg and back to $v_s$ in only $2k_s-1$ rounds. Since $k_s\leq k_j$, $R_s$ can damage every vertex in its leg, including $v_s$, without being captured. Each remaining robber $R_i$ for $i\not=s,j$ also damages $k_i$ vertices and $R_j$ damages $k_j-1$ vertices. Therefore, at least $\left(\sum_{i=1}^s k_i\right)-1$ are damaged.


Next, assume the cop does not capture $R_j$. Since $R_j$ can always maintain a distance of two from the cop, if the cop ever moves into another leg, then $R_j$ can damage $v_j$. After damaging $v_j$, $R_j$ can stop following the cop and move to the leaf in their leg, damaging $k_j$ vertices. Since all other robbers also damaged all of the vertices in their legs (except for $R_s$, which damaged at least $k_s-1$ vertices), damage is at least $\left(\sum_{i=1}^s k_i\right)-1$. If the cop never leaves the leg containing $R_j$, then $R_j$ can maintain a distance of two from the cop until $R_s$ moves from the leaf in their leg and back to $v_s$. Since the cop is on the leg with $R_j$, it follows that $R_s$ can damage $v_s$ without being captured. After this, $R_j$ can move to the leaf in their leg, damaging $k_j-1$ vertices ($v_j$ will not be damaged). Since all other robbers damaged all of the vertices in their legs, damage is at least $\left(\sum_{i=1}^s k_i\right)-1$.


If the cop starts on one of the $\ell-s$ shortest legs, let $R_i$ place on $v_i$ for $1\leq i\leq s$. Regardless of what the cop does, each robber can move towards the end of their leg on each turn, and will not be caught before they damage every vertex in their leg. Therefore, damage is at least $\sum_{i=1}^s k_i$.

Next, let the cop start on one of the $s$ longest legs; specifically, suppose the cop starts on a vertex on the leg of length $k_j$ for some $1\leq j\leq s$. Choose another leg of length $k_t$ for some $1\leq t\leq s$ and $t\not=j$, and consider the path $P$ of length $k_j+k_t+1$ formed by the two legs and the center vertex. Place two robbers on $P$ in the optimal starting positions relative to the cop for a path on $k_j+k_t+1$ vertices. All other robbers $R_i$ for $1\leq i\leq s$ and $i\not=j,t$ should place on $v_i$. Regardless of what the cop does, each $R_i$ can move towards the end of their leg during each round, damaging all $k_i$ vertices in their leg. Meanwhile, as long as the cop remains on $P$, $R_j$ and $R_t$ should follow the strategy for a path of that length, as outlined in the proof of Theorem \ref{thm:path}. If the cop never leaves $P$, the damage on the path is at least $k_j+k_t+1-2$ and total damage is at least $\left(\sum_{i=1}^s k_i\right)-1$.

Now assume that at some point, the cop leaves $P$ and enters another leg. Consider what strategy each robber was employing on the previous turn, when the cop was necessarily on $c$. If neither robber was attempting to remain two vertices away from the cop, then each robber can continue employing their current strategies from the proof of Theorem \ref{thm:path} and they will be able to damage their parts of the path, damaging at least $k_j+k_t-1$ vertices together. Now suppose one of the robbers was attempting to remain two vertices away from the cop on $P$. Without loss of generality, let this robber be $R_t$. Note, in this case, neither robber will have been captured. While the cop is on $c$ or in another leg of $G$, both robbers should act as if the cop is on $c$. Then, $R_t$ is necessarily on $u_t$ and will remain on this vertex as long as the cop doesn't move to $v_j$ or $v_t$, or until $R_j$ damages all vertices on the other leg in $P$, whichever happens first. If the cop moves to $v_j$ or $v_t$, the robbers continue playing their strategy outlined in Theorem \ref{thm:path} until they damage $k_j+k_t-1$ vertices. If $R_j$ damages all the vertices on their side of $c$ first, then $R_t$ can now move to the leaf on the other side of $c$ in $P$. In this case, the two robbers still damage $k_j+k_t-1$ vertices. Therefore, all $s$ cops together damage at least $\left(\sum_{i=1}^s k_i\right)-1$ vertices.

Finally, we consider the case where $k_p=1$ for some $1\leq p\leq s$ and note this implies that $k_i=1$ for all $p\leq i\leq \ell$. Note if $p=1$, all legs have length one. If the cop starts on $c$ and the robbers all start on $v_1,\cdots, v_s$, the cop can capture at most one robber on the first round, so at least $s-1=\left(\sum_{i=1}^s k_i\right)-1$ vertices are damaged. If the cop does not start on $c$, the robbers can start on at least $s-1$ of the vertices $v_1,\cdots, v_s$ and the cop cannot capture a robber on the first round. Thus, at least $s-1=\left(\sum_{i=1}^s k_i\right)-1$ vertices are damaged. 

Now, assume $p \geq 2$, so there exists at least one leg of length at least two. In this case, if the cop starts on a vertex other than $c$, the argument follows as in the case where $k_i\geq 2$ for each $1 \leq i \leq s$. If the cop starts on $c$, let $R_i$ place on $u_i$ for each $1\leq i\leq p-1$ and let $R_i$ place on $v_i$ for each $p\leq i\leq s$. If the cop moves in the first round to a leg of length one (which may or may not contain a robber), the vertex in that leg is not damaged. However, all robbers $R_i$ not contained in that leg can then damage $v_i$ in at most two rounds (moving to do so if necessary) as well as any remaining vertices in their respective legs. So in this case, damage is at least $\left(\sum_{i=1}^s k_i\right)-1$. If the cop moves in the first round to a leg of length at least two, the argument proceeds the same as the $k_i\geq 2$ case. If the cop does not move in the first round, then all robbers $R_i$ for $p\leq i\leq s$ damage the vertex in their leg since they are not captured in this round. Let $R_{p-1}$, the robber on the shortest leg with length at least 2, move towards the leaf in their leg while all robbers $R_j$ such that $1\leq j\leq p-2$ (if such robbers exist) remain still. From here, the argument again follows as in the $k_i\geq 2$ case.

We have shown that for each cop strategy, the $s$ robbers can damage at least $\left(\sum_{i=1}^s k_i\right)-1$ vertices, obtaining the desired result.\end{proof}


\section{Graphs with extreme $2$-robber damage number}\label{sec:extreme2robber}
We now turn our focus to the $2$-robber damage number and consider graphs for which this number is high and low. First, we note that by Proposition \ref{prop:Empty}, $\dmg(\overline{K_1},2)=0=n-1$, $\dmg(\overline{K_2},2)=1=n-1$, $\dmg(\overline{K_3},2)=2=n-1$, and $\dmg(\overline{K_n},2)=2<n-1$ for all $n\geq 4$. We also know that $\dmg(P_2,2)=0=n-2$ by Theorem \ref{thm:path}. Having considered all graphs on one and two vertices and all graphs on three or more vertices without an edge, we can now apply our general bounds from Section \ref{sec:generalBounds} to all other graphs. Proposition \ref{prop:damageAtLeastSMinus1} shows that for every graph $G$ on at least $3$ vertices, $\dmg(G,2)\geq 1$ and Proposition \ref{prop:damageAtMostNMinus2} shows that $\dmg(G,2)\leq n-2$ for every graph on at least $2$ vertices and at least $1$ edge. In this section, we provide a complete characterization of the graphs which achieve these bounds, starting with $\dmg(G,2)=n-2$.

\begin{thm}\label{thm:CharN-2}
Suppose $G$ is a graph on $n$ vertices, 
\[\mathcal{G'} = \{\overline{K_4}, K_2 \cup K_1 \cup K_1, K_2 \cup K_2 \cup K_1, K_2 \cup K_2 \cup K_2, K_2 \cup K_1, K_2 \cup K_2\},\]
and $\mathcal{G} = \{P_n ~|~ n \geq 2\} \cup \{C_n \ | \ n \geq 3\} \cup \mathcal{G'}$.
Then, $\dmg(G; 2) = n-2$ if and only if $G \in \mathcal{G}$.
\end{thm}
\begin{proof}
By Proposition \ref{prop:maxDegreeThree}, if a graph $G$ has $\dmg(G;2) = n-2$, then $\Delta(G) \leq 2$. Therefore, it suffices to determine which graphs with $\Delta(G) \leq 2$ satisfy  $\dmg(G;2) = n-2$. Observe that if $\Delta(G) \leq 2$, then every component of $G$ is a path or cycle. We can divide these graphs into cases based on the number of components in $G$. First, consider the case where $G$ has at least five components. This implies that there are at least three components that do not contain a robber and therefore, at least three vertices remain undamaged. Thus, any graph with five or more components must not satisfy $\dmg(G; 2) = n-2$. 

Next, suppose $G$ has four components and there exists a component with at least two vertices (i.e., $G \ncong \overline{K_4}$). If the cop starts on the component with at least two vertices and protects an edge, they prevent two vertices from being damaged. Also, since there are four components and only three players, there is always at least one empty component which implies that at least one additional vertex will be undamaged. Therefore, $\dmg(G; 2) \leq n-3$. Finally, the only other graph to consider with four components is $\overline{K_4}$. By Proposition \ref{prop:Empty}, $\dmg(\overline{K_4};2)=2=n-2$.

Now, suppose $G$ has three components and at least one component contains three or more vertices. Assume the cop starts on a component that contains at least three vertices. We again show that no matter what the robbers do, the cop can prevent damage to at least three vertices. First, suppose the robbers start on separate components and not the one with the cop. This implies that the cop can protect their entire component and therefore, at least three vertices are protected. Next, suppose one robber starts on the same component as the cop and the other robber starts on another component. This implies that the third component remains undamaged and hence, at least one vertex in that component is protected. We also know that the cop can protect an edge in their component and therefore, at least two more vertices are protected. Thus, at least three vertices in total remain undamaged. Finally, suppose both robbers start on the component with the cop. Again, we know the cop can protect an edge, preventing two vertices from being damaged in their component. Since there are two more components, we know at least four vertices can be protected in total. This implies that $\dmg(G; 2) \leq n-3$, regardless of the robbers' strategy. 

The only remaining graphs to consider with three components have at most two vertices in each component. This includes four graphs: $K_2 \cup K_1 \cup K_1, K_2 \cup K_2 \cup K_1, K_2 \cup K_2 \cup K_2, \overline{K_3}$.  Consider the first three graphs, each of which contains $K_2$ as a component. In this case, it is optimal for the cop to start on $K_2$. Then, it is clear that the cop can protect two vertices and the robbers can always damage the remaining $n-2$ vertices by starting on separate components. Therefore, $\dmg(G; 2) = n-2$ for these three graphs. Clearly, if  $G \cong \overline{K_3}$, $\dmg(G; 2) = 2 = n-1$.

Next, suppose $G$ has two components and one component contains at least four vertices. We know that this component, $G_1$, must be a path or cycle on four or more vertices. Let the cop start on $G_1$. First, note that if both robbers start on $G_1$, at most $n-3$ vertices are damaged because $\dmg(G_1; 2) = |G_1|-2 < n-2$. Next, let one robber start on $G_1$ and let the other robber start on the other component, $G_2$. Since $G_1$ is a path or a cycle and $|G_1| \geq 4$, we know that $\dmg(G_1) = \lfloor\frac{|G_1|-1}{2}\rfloor \leq \frac{|G_1|-1}{2}$ \cite{CS19}. Note that if $\frac{|G_1|-1}{2} > |G_1|-3$, then $|G_1| < 5$. This implies that if $|G_1| \geq 5$, $\lfloor\frac{|G_1|-1}{2}\rfloor \leq \frac{|G_1|-1}{2} \leq |G_1|-3$. Furthermore, if $|G_1|=4$, then $\lfloor \frac{|G_1|-1}{2} \rfloor = 1 = |G_1|-3$. Therefore in this case,  $\dmg(G_1) \leq |G_1|-3$ because $|G_1| \geq 4$. This implies that $\dmg(G; 2) \leq |G_2| + |G_1| - 3 = n-3$. 

Now, we consider graphs $G$ with two components where both components contain fewer than four vertices. Suppose one of the components of $G$ has at least three vertices. This implies that one component (say $G_1$) is either $C_3$ or $P_3$. Let the cop start on $G_1$. First, let one robber start on $G_1$ and let the other robber start on the other component, $G_2$. We know that the damage number of $P_3$ and $C_3$ is zero and therefore, at least three vertices are protected. If both robbers start on $G_1$, then two vertices are protected in the first component (since $\dmg(G_1; 2) = 1$) and at least one vertex remains undamaged from the second component. Finally, if both robbers start on $G_2$, it is clear that $|G_1| = 3$ vertices are protected. Thus, $\dmg(G; 2) \leq n-3$. 

Next, suppose both components of $G$ contain at most two vertices. This means $G$ is either $K_2 \cup K_1$, $K_2 \cup K_2$ or $\overline{K_2}$. For the first two graphs, we know the cop can protect an edge and therefore protect a $K_2$ which means the robbers can only damage the vertices in the other component. This implies that $\dmg(K_2 \cup K_1; 2) = 1 = n-2$ and $\dmg(K_2 \cup K_2; 2) = 2 = n-2$. If $G = \overline{K_2}$, then $\dmg(G; 2) = 1 = n-1$. 

Finally, we consider graphs $G$ with one component. In this case, $G$ is either $P_n$ or $C_n$. From Theorem \ref{thm:path}, we know that for $n \geq 2$, $\dmg(P_n; 2) = n-2$ and from Theorem \ref{thm:cycle}, we know that for $n \geq 3$, $\dmg(C_n; 2) = n-2$.
\end{proof}

Turning our attention to graphs with $2$-robber damage number equal to one, we first characterize such graphs on at most four vertices.

\begin{prop}
Let $G$ be a graph on $n\leq 4$ vertices. Then, $\dmg(G,2)=1$ if and only if $G$ is not one of the graphs in $\{P_2,P_4,C_4,\overline{K_1},\overline{K_3},\overline{K_4},K_2\cup K_2, K_2\cup K_1 \cup K_1\}$.
\end{prop}

\begin{proof}
By Theorem \ref{thm:CharN-2}, $\dmg(P_2,2)=0$, $\dmg(P_4,2)=2$, $\dmg(C_4,2)=2$, $\dmg(\overline{K_4},2)=2$, $\dmg(K_2\cup K_2,2)=2$, and $\dmg(K_2\cup K_1 \cup K_1,2)=2$. By Proposition \ref{prop:Empty}, $\dmg(\overline{K_1},2)=0$ and $\dmg(\overline{K_3},2)=2$.

Let $G$ be a graph on four vertices which is not $P_4,C_4,\overline{K_4},K_2\cup K_2,$ or $K_2\cup K_1 \cup K_1$. Then, by Proposition \ref{prop:damageAtMostNMinus2} and Theorem \ref{thm:CharN-2}, $\dmg(G,2)\leq 1$. By Proposition \ref{prop:damageAtLeastSMinus1}, $\dmg(G,2)\geq 1$, so $\dmg(G,2)=1$. Let $G$ be a graph on three vertices which is not $\overline{K_3}$. Then, by Proposition \ref{prop:damageAtMostNMinus2}, $\dmg(G,2)\leq 1$ and by Proposition \ref{prop:damageAtLeastSMinus1}, $\dmg(G,2)\geq 1$, so $\dmg(G,2)=1$. Finally, the only graph on two vertices which is not $P_2$ is $\overline{K_2}$, and by Proposition \ref{prop:Empty}, $\dmg(\overline{K_2},2)=1$.
\end{proof}

Finally, we show that a graph $G$ on five or more vertices has $\dmg(G,2)=1$ if and only if it is a threshold graph. A graph is a \emph{threshold graph} if it can be constructed from a single vertex by repeated additions of a single isolated vertex or a single dominating vertex. Alternatively, it is well known that threshold graphs can be characterized by their forbidden subgraphs: a graph is threshold if and only if no four of its vertices contain $P_4$, $C_4$, or $K_2\cup K_2$ as an induced subgraph. Both of these characterizations are crucial in the following argument. 

\begin{thm}
Let $G$ be a graph on $n\geq 5$ vertices. Then, $\dmg(G,2)=1$ if and only if $G$ is a threshold graph with no more than one isolated vertex.
\end{thm}
\begin{proof}
We will prove the contrapositive. First, assume $G$ is not a threshold graph. Therefore, $G$ contains at least one of the following as an induced subgraph on four vertices: $P_4$, $C_4$, or $K_2\cup K_2$. Let $A$ be the induced subgraph of $G$ which is isomorphic to one of these graphs and label the vertices of $A$ as shown in Figure \ref{fig:ForbiddenSub}.

\begin{figure}[h!]
    \centering
    \includegraphics{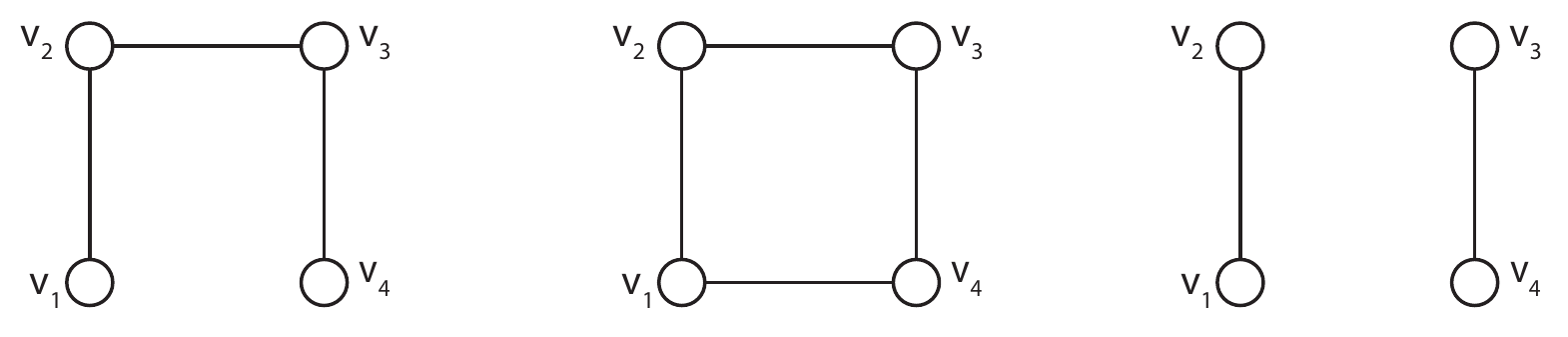}
    \caption{Vertex labelings for $P_4$, $C_4$, and $K_2\cup K_2$, respectively}
    \label{fig:ForbiddenSub}
\end{figure}

If the cop's initial placement is $v_1$, let the robbers initially place on $v_3$ and $v_4$. If $A$ is isomorphic to $P_4$ or $K_2\cup K_2$, the robbers evade capture during the first round. Thus, at least two vertices are damaged. If $A$ is isomorphic to $C_4$, the robber on $v_4$ is captured in the first round, but the robber on $v_3$ can damage $v_3$ and move to $v_2$ in this round. In the second round, the cop cannot capture the remaining robber since $v_2v_4\not\in E(G)$. Thus, the robber also damages $v_2$, so at least two vertices are damaged. If the cop's initial placement is $v_4$, at least two vertices are damaged by a similar argument.

If the cop's initial placement is a vertex $v$ other than $v_1$ and $v_4$, let the robbers initially place on $v_1$ and $v_4$. If $v$ is not adjacent to at least one of $v_1$ or $v_4$, then both robbers evade capture and damage their initial vertices in the first round. Thus, at least two vertices are damaged. 

If $v$ is adjacent to at least one of $v_1$ or $v_4$, without loss of generality, assume $v$ is adjacent to $v_1$. Then, during the first round, the cop can capture the robber on $v_1$. The robber on $v_4$ can damage $v_4$ and move to $v_3$ during round one. Since $v_1v_3\not\in E(G)$, the cop cannot capture the remaining robber during round two. Thus, the robber also damages $v_3$ and at least two vertices are damaged. Since the robbers can damage at least two vertices for any cop strategy, it is clear that $\dmg(G,2)\geq 2$.

Next, assume that $G$ is a threshold graph with no more than one isolated vertex. Since threshold graphs are formed by repeated additions of a single isolated vertex or a single dominating vertex, a threshold graph with no isolated vertices has a dominating vertex and a threshold graph with one isolated vertex has a vertex of degree $n-2$. In either case, let $v$ be this vertex of high degree ($n-1$ or $n-2$).

Assume for contradiction that $\dmg(G,2)\geq 2$ and consider the case where the cop initially places on $v$. If there is an isolated vertex in $G$, then the robbers must not place here since the robber on the isolated vertex would only damage one vertex and the other robber would be captured in the first round, contradicting $\dmg(G,2)\geq 2$. Therefore, the robbers both initially place on vertices adjacent to $v$. If the robbers initially place on the same vertex, they will both be captured in the first round, again contradicting $\dmg(G,2)\geq 2$. So, the robbers must initially place on two distinct vertices, $v_1$ and $v_4$.

Since the cop can always capture one robber during round one and $\dmg(G,2)\geq 2$, the remaining robber must be able to damage their initial vertex and move to a neighboring vertex. In order for this robber to damage a second vertex, they need to move to a neighboring vertex which is not adjacent to the cop's location. Thus, it must be that $v_1$ has a neighbor which is not adjacent to $v_4$ and $v_4$ has a neighbor which is not adjacent to $v_1$. Let $v_2$ and $v_4$ be the neighbors of $v_1$ and $v_3$, respectively. Since $v_1v_3,v_2v_4\not\in E(G)$ and $v_1v_2,v_3v_4\in E(G)$, the induced subgraph of $G$ on vertices $\{v_1,v_2,v_3,v_4\}$ is one of the graphs in Figure \ref{fig:ForbiddenSub}. However, since $G$ is a threshold graph, it does not contain any of these graphs as induced subgraphs. Therefore, it must be that $\dmg(G,2)=1$, as desired.
\end{proof}

\section{Concluding Remarks}
Many of the results in this paper focus on damage in the most general case with $s$ robbers. However, there are a few results that we proved for the $s=2$ case that we were unable to generalize. For the first such result, Proposition \ref{prop:maxDegreeThree}, we demonstrate that one potential generalization does not hold and in Conjecture \ref{conj:maxdeg}, we provide another generalization. 

In Section \ref{sec:extreme2robber}, we characterize the graphs with the highest and lowest $2$-robber damage number. In general, we know by Proposition \ref{prop:damageAtLeastSMinus1} and \ref{prop:damageAtMostNMinus2} that for a non-empty graph $G$ on $n\geq 2$ vertices, the maximum $s$-robber damage number is $n-2$ and the minimum is $s-1$. By Theorems \ref{thm:path}, \ref{thm:cycle}, and \ref{thm:star}, we already know that $\dmg(G;s)=n-2$ for paths, cycles, and spiders with no more than $s$ legs. However, this is not a complete characterization. It's worth noting that Proposition \ref{prop:maxDegreeThree} is used in the proof of Theorem \ref{thm:CharN-2}, which characterizes the graphs whose $2$-robber damage number is $n-2$. Therefore, it is possible that generalizing Proposition \ref{prop:maxDegreeThree}, possibly by proving Conjecture \ref{conj:maxdeg}, may lead to a characterization of graphs with $\dmg(G;s)=n-2$.



\end{document}